\documentclass[11pt]{article}
\usepackage{amsmath}
\usepackage{amsthm}
\usepackage{amssymb}
\usepackage{tikz}
\usepackage{hyperref}
\usepackage[left=3.5cm,right=3.5cm,top=3cm,bottom=3cm]{geometry}
\newtheorem{definition}{Definition}
\newtheorem{theorem}{Theorem}
\newtheorem{proposition}{Proposition}
\newtheorem{corollary}{Corollary}

\newtheorem{lemma}{Lemma}

\DeclareMathOperator{\GL}{GL}
\DeclareMathOperator{\SL}{SL}
\DeclareMathOperator{\Lie}{Lie}
\DeclareMathOperator{\Adj}{Ad}

\DeclareMathOperator{\aker}{ker}
\DeclareMathOperator{\adj}{ad}

\DeclareMathOperator{\Hom}{Hom}
\DeclareMathOperator{\vspan}{span}

\DeclareMathOperator{\Int}{Int}

\DeclareMathOperator{\Sym}{Sym}

\DeclareMathOperator{\Un}{U}
\DeclareMathOperator{\Map}{Map}

\begin{document}
\title{The Torus-Equivariant Cohomology of Nilpotent Orbits}
\author{Peter Crooks\thanks{Department of Mathematics. University of Toronto. Toronto, ON, Canada. \newline {\tt peter.crooks@utoronto.ca} \newline The author was supported by an NSERC CGS-D3 research grant.}}

 \maketitle
\begin{abstract}
We consider aspects of the geometry and topology of nilpotent orbits in finite-dimensional complex simple Lie algebras. In particular, we give the equivariant cohomologies of the regular and minimal nilpotent orbits with respect to the action of a maximal compact torus of the overall group in question.
\end{abstract}

\section{Introduction}\label{intro}
\subsection{Generalities}
Throughout, we let $G$ be a connected, simply-connected complex simple linear algebraic group. Let $K\subseteq G$ be a maximal compact subgroup, and fix a maximal torus $T\subseteq K$. Set $H:=T_{\mathbb{C}}$, a maximal torus of $G$. Denote by $\mathfrak{g}$, $\mathfrak{k}$, $\mathfrak{t}$, and $\mathfrak{h}$ the Lie algebras of $G$, $K$, $T$, and $H$, respectively. Let $W=N_K(T)/T=N_G(H)/H$ be the Weyl group. Also, let $\Adj:G\rightarrow\GL(\mathfrak{g})$ and $\adj:\mathfrak{g}\rightarrow\mathfrak{gl}(\mathfrak{g})$ denote the adjoint representations of $G$ and $\mathfrak{g}$, respectively. Let $\Delta\subseteq\Hom(T,\Un(1))=\Hom(H,\mathbb{C}^*)$ denote the resulting collection of roots of $\mathfrak{g}$ with respect to the adjoint representation of $T$. By fixing a Borel subgroup $B\subseteq G$ containing $H$, we specify collections $\Delta_{+},\Delta_{-}\subseteq\Delta$ of positive and negative roots, respectively. Let $\Pi\subseteq\Delta_{+}$ denote the resulting collection of simple roots.

Recall that a point $\xi\in\mathfrak{g}$ is called \textit{nilpotent} if the vector space endomorphism $\adj_{\xi}:\mathfrak{g}\rightarrow\mathfrak{g}$ is nilpotent. Recall also that the nilpotent cone is the closed subvariety $\mathcal{N}$ of $\mathfrak{g}$ consisting of the nilpotent elements. We call an adjoint $G$-orbit a \textit{nilpotent orbit} if it is contained in $\mathcal{N}$. As an orbit of an algebraic $G$-action, any nilpotent orbit is a smooth locally closed subvariety of $\mathfrak{g}$.

It is well-known that there exist only finitely many nilpotent orbits of $G$. Indeed, if $G=\SL_n(\mathbb{C})$, then one can use Jordan canonical forms to give an explicit indexing of the nilpotent orbits by the partitions of $n$.

Furthermore, the nilpotent orbits constitute an algebraic stratification of $\mathcal{N}$ (see \cite{RTCG}). In other words, we have the partial order on the set of nilpotent orbits given by $\Theta_1\leq\Theta_2$ if and only if $\Theta_1\subseteq\overline{\Theta_2}$ (the Zariski-closure of $\Theta_2$ in $\mathcal{N}$). Hence, $$\overline{\Theta}=\bigcup_{\Omega\leq\Theta}\Omega$$ for all nilpotent orbits $\Theta$.

It turns out that the set of nilpotent orbits has a unique maximal element, $\Theta_{\text{reg}}$, and a unique minimal non-zero element, $\Theta_{\text{min}}$. These distinguished orbits are called the \emph{regular} and \emph{minimal} nilpotent orbits, respectively. The former consists precisely of the regular nilpotent elements of $\frak{g}$, while the latter is the orbit of a root vector for a long root.

\subsection{Context}
The study of nilpotent orbits lies at the interface of algebraic geometry, representation theory, and symplectic geometry. Indeed, one has the famous Springer resolution $$\mu:T^*(G/B)\rightarrow\mathcal{N}$$ of the singular nilpotent cone (see \cite{RTCG}). The fibres of $\mu$ over a given nilpotent orbit $\Theta$ are isomorphic as complex varieties, and this isomorphism class is called the Springer fibre of $\Theta$. The Springer correspondence then gives a realization of the irreducible complex $W$-representations on the Borel-Moore homology groups of the Springer fibres (see \cite{RTCG}).

From the symplectic standpoint, we note that coadjoint $G$-orbits are canonically complex symplectic manifolds. Since the Killing form on $\frak{g}$ provides an isomorphism between the adjoint and coadjoint representations of $G$, it follows that adjoint $G$-orbits (and in particular, nilpotent $G$-orbits) are naturally complex symplectic manifolds.

Some attention has also been given to the matter of computing topological invariants of nilpotent orbits. In \cite{Collingwood}, Collingwood and McGovern compute the fundamental group of each nilpotent orbit in the classical Lie algebras. Also, Juteau's paper \cite{Juteau} gives the integral cohomology groups of the minimal nilpotent orbit in each of the finite-dimensional complex simple Lie algebras. Additionally, Biswas and Chatterjee compute $H^2(\Theta;\mathbb{R})$ for $\Theta$ any nilpotent orbit in a finite-dimensional complex simple Lie algebra (see their paper \cite{Biswas}).

Our contribution is a computation of the $T$-equivariant cohomology algebras of the $G$-orbits $\Theta_{\text{reg}}$ and $\Theta_{\text{min}}$. (To this end, $H_T^*(X)$ shall always denote the $T$-equivariant cohomology over $\mathbb{Q}$ of a $T$-manifold $X$.) We state our result below.

\begin{theorem}
\begin{itemize}
\item[(i)] $H_T^*(\Theta_{\text{reg}})\cong H^*(G/B;\mathbb{Q})$
\item[(ii)] Let $\alpha\in\Delta_{+}$ be the highest root, and let $\Xi:=\{\beta\in\Pi:\langle\alpha,\beta\rangle=0\}$. Let $W_{\Xi}$ be the subgroup of $W$ generated by the reflections $s_{\beta}$, $\beta\in\Xi$. Then, $H_T^*(\Theta_{\text{min}})$ is isomorphic to the quotient of
    \begin{align} \{& f\in\Map(W/W_{\Xi},H_{T}^*(\text{pt})): (w\cdot\beta)\vert(f([w])-f([ws_{\beta}]))\nonumber\\ & \forall{w\in W},\beta\in\Delta_{-},\text{ }\langle\alpha,\beta\rangle\neq 0\}\nonumber \end{align}
    by the ideal generated by the map $W/W_{\Xi}\rightarrow H_T^*(\text{pt})$, $[w]\mapsto w\cdot\alpha$.
\end{itemize}
\end{theorem}

\subsection{Structure of the Article}
Section \ref{regular} is devoted to an examination of the regular nilpotent orbit. Specifically, we establish a few facts concerning the structure of the $G$-stabilizer $C_G(\eta)$ of a point $\eta\in\Theta_{\text{reg}}$. We then give a new description of $\Theta_{\text{reg}}\cong G/C_G(\eta)$ as a $T$-manifold (see Theorem \ref{diffeomorphism}). This description is suitable for purposes of computing $H_T^*(\Theta_{\text{reg}})$.

Section \ref{minimal} treats the case of the minimal nilpotent orbit, but the approach differs considerably from that adopted when studying $\Theta_{\text{reg}}$. We begin by introducing a natural $\mathbb{C}^*$-action on nilpotent orbits. Via this action, we define $\mathbb{P}(\Theta_{\text{min}})$, a smooth closed subvariety of $\mathbb{P}(\frak{g})$. This variety has interesting properties beyond those materially relevant to computing $H_T^*(\Theta_{\text{min}})$. In particular, $\mathbb{P}(\Theta_{\text{min}})$ is naturally a symplectic manifold, and the $T$-action on $\Theta_{\text{min}}$ descends to a Hamiltonian action on $\mathbb{P}(\Theta_{\text{min}})$. Accordingly, we give an explicit description of $\mathbb{P}(\Theta_{\text{min}})^T$ (see \ref{Fixed}) and use it to find the moment polytope of $\mathbb{P}(\Theta_{\text{min}})$ (see \ref{Polytope}).

In \ref{Partial Flag}, we use GKM Theory to provide a description of $H_T^*(G/P)$, where $P\subseteq G$ is a parabolic subgroup containing $T$. This is done in recognition of the fact (which we prove in \ref{Description}) that $\mathbb{P}(\Theta_{\text{min}})$ is $G$-equivariantly isomorphic to $G/P_{\Xi}$, where $P_{\Xi}$ is the parabolic determined by $\Xi$.

It then remains to relate the graded algebras $H_T^*(G/P_{\Xi})$ and $H_T^*(\Theta_{\text{min}})$. This is achieved via the Thom-Gysin sequence in $T$-equivariant cohomology, which allows us to exhibit $H_T^*(\Theta_{\text{min}})$ as a quotient of $H_T^*(G/P_{\Xi})$. Indeed, we take the quotient of $H_T^*(G/P_{\Xi})$ by the ideal generated by the $T$-equivariant Euler class of the associated line bundle $G\times_{P_{\Xi}}\frak{g}_{\alpha}\rightarrow G/P_{\Xi}$.

\subsection*{Acknowledgements}
I would like to begin by thanking my doctoral thesis supervisor, Lisa Jeffrey. I have benefitted considerably from her time, attention, and support. I am also grateful to Faisal Al-Faisal, Jonathan Fisher, Allen Knutson, Eckhard Meinrenken, Robert Milson, and Daniel Rowe for fruitful conversations. Finally, I am grateful to the Natural Sciences and Engineering Research Council of Canada (NSERC) for the financial support I have received during the doctoral programme.
\section{The Regular Nilpotent Orbit}\label{regular}
Throughout this section, we may actually take $G$ to be semisimple. Now, recall that an element $\xi\in\mathfrak{g}$ is called \textit{regular} if the dimension of the Lie algebra centralizer $C_{\mathfrak{g}}(\xi)=\{X\in\mathfrak{g}:[X,\xi]=0\}$ coincides with the rank of $\mathfrak{g}$. The regular nilpotent elements of $\mathfrak{g}$ actually constitute $\Theta_{\text{reg}}$.

Let us construct a reasonably standard representative of $\Theta_{\text{reg}}$. Indeed, for each $\beta\in\prod$, choose a root vector $e_{\beta}\in\mathfrak{g}_{\beta}\setminus\{0\}$. Consider the nilpotent element $$\eta:=\sum_{\beta\in\prod} e_{\beta}.$$ In \cite{KOST}, Kostant proved that $\eta\in\Theta_{\text{reg}}$. Furthermore, one can easily prove that $C_{\mathfrak{g}}(\eta)$ belongs to the positive nilpotent subalgebra $\mathfrak{n}_+:=\bigoplus_{\beta\in\Delta^+}\mathfrak{g}_{\beta}$.

Let $C_G(\eta)=\{g\in G:\Adj_g(\eta)=\eta\}$ be the $G$-stabilizer of $\eta$. This gives an isomorphism $\Theta_{\text{reg}}\cong G/C_G(\eta)$ of complex $G$-varieties, where the action of $C_G(\eta)$ on $G$ is given by $x:g\mapsto gx^{-1}$, $x\in C_G(\eta)$, $g\in G$.

Having realized $\Theta_{\text{reg}}$ in this way, we turn our attention to $C_G(\eta)$. To this end, we recall that the inner automorphism group (or adjoint group) of $\mathfrak{g}$ is the subgroup $\Int(\mathfrak{g})$ of $\GL(\mathfrak{g})$ generated by all automorphisms of the form $e^{\adj_{\xi}}$, $\xi\in\mathfrak{g}$. Since $\Adj_{\exp(\xi)}=e^{\adj_{\xi}}$ for all $\xi\in\mathfrak{g}$, it follows that $\Int(\mathfrak{g})$ is precisely the image of the adjoint representation $\Adj:G\rightarrow\GL(\mathfrak{g})$. Hence, $\Int(\mathfrak{g})$ is a connected Zariski-closed subgroup of $\GL(\mathfrak{g})$.

We shall require the below theorem concerning the structure of the $\Int(\mathfrak{g})$-stabilizer $C_{\Int(\mathfrak{g})}(\eta)$ of $\eta$.

\begin{theorem}\label{stabthm}
The centralizer $C_{\Int(\mathfrak{g})}(\eta)$ is a connected abelian unipotent subgroup of $\Int(\mathfrak{g})$.
\end{theorem}

This is Theorem 2.6 in \cite{LGLA3}.

We note that a connected unipotent complex linear algebraic group is isomorphic to affine space as a variety. For our purposes, the relevant observation is that $C_{\Int(\mathfrak{g})}(\eta)$ is isomorphic as a complex manifold to $\mathbb{C}^n$ for some $n$.

\begin{proposition}\label{stabcentre}
The inclusion of the centre $Z(G)\hookrightarrow C_G(\eta)$ is a homotopy-equivalence.
\end{proposition}

\begin{proof}
Note that $\phi:C_G(\eta)\rightarrow C_{\Int(\mathfrak{g})}(\eta)$, $g\mapsto\Adj_g$, is a surjective Lie group morphism. Since $G$ is connected, $Z(G)$ is the kernel of the adjoint representation, and hence is also the kernel of $\phi$. This yields a fibre bundle $$Z(G)\rightarrow C_G(\eta)\xrightarrow{\phi} C_{\Int(\mathfrak{g})}(\eta).$$ Since the base space $C_{\Int(\mathfrak{g})}(\eta)$ is contractible, our bundle is trivial. Noting that the inclusion of a fibre in a trivial bundle over $C_{\Int(\mathfrak{g})}(\eta)$ is a homotopy-equivalence, the inclusion $Z(G)\hookrightarrow C_G(\eta)$ is also a homotopy-equivalence.
\end{proof}

An immediate corollary of Proposition \ref{stabcentre} is the existence of an isomorphism of graded $\mathbb{Z}$-algebras between the $G$-equivariant cohomology $H_G^{*}(\Theta_{\text{reg}};\mathbb{Z})$ of the regular nilpotent orbit and the group cohomology $H^{*}_{\text{gp}}(Z(G);\mathbb{Z})$ of the finite group $Z(G)$.

\begin{corollary}
$H_G^{*}(\Theta_{\text{reg}};\mathbb{Z})\cong H^{*}_{\text{gp}}(Z(G);\mathbb{Z})$
\end{corollary}

\begin{proof}
Recall that $$H_G^{*}(\Theta_{\text{reg}};\mathbb{Z})\cong H_G^{*}(G/C_G(\eta);\mathbb{Z})=H^{*}((G/C_G(\eta))_G;\mathbb{Z}),$$ where $(G/C_G(\eta))_G$ is the quotient of $EG\times(G/C_G(\eta))$ by the diagonal action of $G$. This quotient is homeomorphic to $EG/C_G(\eta)$, and hence $$H_G^{*}(\Theta_{\text{reg}};\mathbb{Z})\cong H^{*}(EG/C_G(\eta);\mathbb{Z}).$$

By Proposition \ref{stabcentre}, $EG/Z(G)\rightarrow EG/C_G(\eta)$ is a fibre bundle with contractible fibre $C_G(\eta)/Z(G)$. Hence, $$H^{*}(EG/C_G(\eta);\mathbb{Z})\cong H^{*}(EG/Z(G);\mathbb{Z}).$$ However, we may take $EG/Z(G)$ to be the classifying space $BZ(G)$, whose singular cohomology coincides with the group cohomology $H^{^*}_{\text{gp}}(Z(G);\mathbb{Z})$. This completes the proof.
\end{proof}

\begin{corollary}
\begin{itemize}
\item[(i)] There is a natural complex Lie group isomorphism $C_G(\eta)_0\cong C_{\Int(\mathfrak{g})}(\eta)$, where $C_G(\eta)_0$ is the identity component of $C_G(\eta)$.
\item[(ii)] There is a natural central extension $$1\rightarrow Z(G)\rightarrow C_G(\eta)\rightarrow C_G(\eta)_0\rightarrow 1,$$ and the inclusion $C_G(\eta)_0\rightarrow C_G(\eta)$ is a splitting. In particular, $C_G(\eta)$ is the internal direct product $Z(G)\times C_G(\eta)_0$.
\end{itemize}
\end{corollary}

\begin{proof}
Since $Z(G)\rightarrow C_G(\eta)$ is a homotopy equivalence, it induces a group isomorphism $$\pi_0(Z(G))\xrightarrow{\cong}\pi_0(C_G(\eta))\cong C_G(\eta)/C_G(\eta)_0.$$ Hence, if $\omega\in C_G(\eta)/Z(G)$ is a coset, then there exists a unique $g\in C_G(\eta)_0$ for which $[g]=\omega$. Now, recall that $C_G(\eta)/Z(G)\rightarrow C_{\Int(\mathfrak{g})}(\eta)$, $[g]\mapsto\Adj_g$, is an isomorphism.  Hence, if $f\in C_{\Int(\mathfrak{g})}(\eta)$, then there exists a unique $g\in C_G(\eta)_0$ for which $\Adj_g=f$. Accordingly, $\varphi:C_{\Int(\mathfrak{g})}(\eta)\rightarrow C_G(\eta)_0$, $\Adj_g\mapsto g$, $g\in C_G(\eta)_0$, is a well-defined complex Lie group isomorphism.

For the second part, note that one always has the central extension $$1\rightarrow Z(G)\rightarrow C_G(\eta)\xrightarrow{\pi} C_{\Int(\mathfrak{g})}(\eta)\rightarrow 1,$$ where $\pi:C_G(\eta)\rightarrow C_{\Int(\mathfrak{g})}(\eta)$ is the projection map. By replacing $C_{\Int(\mathfrak{g})}(\eta)$ with the isomorphic copy $C_G(\eta)_0$ and setting $\psi:=\varphi\circ\pi:C_G(\eta)\rightarrow C_G(\eta)_0$, we obtain the central extension $$1\rightarrow Z(G)\rightarrow C_G(\eta)\xrightarrow{\psi} C_G(\eta)_0\rightarrow 1.$$ Note that the inclusion $C_G(\eta)_0\rightarrow C_G(\eta)$ splits this sequence.
\end{proof}

\begin{theorem}\label{diffeomorphism}
The regular nilpotent orbit $\Theta_{\text{reg}}$ is $T$-equivariantly diffeomorphic to a product $K/Z(G)\times V$, where $V$ is a finite-dimensional real vector space on which $T$ acts trivially.
\end{theorem}

\begin{proof}
Earlier, we noted that the centralizer $C_{\mathfrak{g}}(\eta)$ belonged to the nilpotent subalgebra $\mathfrak{n}_+$. Letting $N$ denote the connected closed subgroup of $G$ with Lie algebra $\mathfrak{n}_+$, this fact implies the inclusion $C_G(\eta)_0\subseteq N$.

Letting $A$ denote the connected closed subgroup of $G$ with (real) Lie algebra $i\mathfrak{t}\subseteq\mathfrak{g}$, the Iwasawa decomposition gives a diffeomorphism $\Phi:K\times A\times N\xrightarrow{\cong} G$, $$(k,a,n)\mapsto kan.$$ Now, let $Z(G)\times C_G(\eta)_0$ act on $G$ via the $C_G(\eta)$-action on $G$ and the isomorphism $Z(G)\times C_G(\eta)_0\rightarrow C_G(\eta)$. Explicitly, this action is given by $$(z,h):g\mapsto g(zh)^{-1},$$ $(z,h)\in Z(G)\times C_G(\eta)_0$, $g\in G$. We enlarge this to a $T\times(Z(G)\times C_G(e)_0)$-action with $T$ acting on $G$ by left-multiplication. Note that $\Phi$ is then a $T\times (Z(G)\times C_G(\eta)_0)$-manifold isomorphism for the action of $T\times (Z(G)\times C_G(\eta)_0)$ on $K\times A\times N$ defined by $$(t,z,h):(k,a,n)\mapsto(tkz^{-1},a,nh^{-1}),$$ $(t,z,h)\in T\times (Z(G)\times C_G(\eta)_0)$, $(k,a,n)\in K\times A\times N$. It follows that $$\Theta_{\text{reg}}\cong G/C_G(\eta)=G/(Z(G)\times C_G(\eta)_0)$$ is $T$-equivariantly diffeomorphic to the quotient $$(K\times A\times N)/(Z(G)\times C_G(\eta)_0),$$ endowed with its residual $T$-action. The latter is clearly $T$-equivariantly diffeomorphic to $$K/Z(G)\times A\times N/C_G(\eta)_0,$$ where $T$ acts by left-multiplication on the factor $K/Z(G)$ and trivially on the factors $A$ and $N/C_G(\eta)_0$. Since $A$ is diffeomorphic to its Lie algebra, it remains only to establish that $N/C_G(\eta)_0$ is diffeomorphic to a real vector space. However, this follows from the fact that a quotient of a nilpotent connected simply-connected Lie group by a connected closed subgroup is diffeomorphic to a real vector space (see \cite{LGLA3}).
\end{proof}

\begin{corollary}\label{flagvariety}
There is an isomorphism $H_T^*(\Theta_{\text{reg}})\cong H^*(G/B;\mathbb{Q})$.
\end{corollary}

\begin{proof}
By Theorem \ref{diffeomorphism}, $H_T^*(\Theta_{\text{reg}})\cong H_T^*(K/Z(G))$. Since $Z(G)$ is a finite group, the action of $T$ on $K/Z(G)$ is locally free, and $$H_T^*(K/Z(G))\cong H^*(T\backslash K/Z(G);\mathbb{Q})\cong H^*(T\backslash K;\mathbb{Q})\cong H^*(G/B;\mathbb{Q}).$$
\end{proof}

\section{The Minimal Nilpotent Orbit}\label{minimal}
\subsection{A $\mathbb{C}^*$-Action on Nilpotent Orbits}\label{Action}
Fix a non-zero nilpotent orbit $\Theta\subseteq\mathfrak{g}$ and a point $\xi\in\Theta$. By the Jacobson-Morozov Theorem, there exist a semisimple element $h\in\mathfrak{g}$ and a nilpotent element $f\in\mathfrak{g}$ for which $(\xi,h,f)$ is an $\mathfrak{sl}_2(\mathbb{C})$-triple with nil-positive element $\xi$. We note that for all $\lambda\in\mathbb{C}$, $$\Adj_{\exp(\lambda h)}(\xi)=e^{\adj_{\lambda h}}(\xi)=e^{2\lambda}\xi.$$ From this calculation, it follows that $\Theta$ is invariant under the scaling action of $\mathbb{C}^*$ on $\mathfrak{g}$. Accordingly, we introduce $$\mathbb{P}(\Theta):=\Theta/\mathbb{C}^{*},$$ a smooth quasi-projective subvariety of $\mathbb{P}(\mathfrak{g})$. Since the actions of $G$ and $\mathbb{C}^*$ on $\mathfrak{g}$ commute, the $G$-action descends to the quotients $\mathbb{P}(\Theta)$ and $\mathbb{P}(\mathfrak{g})$.

We remark that $\mathbb{P}(\Theta)$ has a rich geometric structure. To see this, choose a $K$-invariant Hermitian inner product $\langle,\rangle:\mathfrak{g}\otimes_{\mathbb{R}}\mathfrak{g}\rightarrow\mathbb{C}$. This yields a $K$-invariant K\"{a}hler structure on $\mathbb{P}(\mathfrak{g})$. Since the usual action of $\Un(n+1)$ on $\mathbb{P}^n$ is Hamiltonian, so too is the action of $K$ on $\mathbb{P}(\mathfrak{g})$. Furthermore, one has the moment map $\Phi:\mathbb{P}(\mathfrak{g})\rightarrow\mathfrak{k}^*$ defined by $$\Phi([\xi])(X)=\frac{\text{Im}(\langle[X,\xi],\xi\rangle)}{\langle\xi,\xi\rangle},$$ where $X\in\mathfrak{g}\setminus\{0\}$ and $\eta\in\mathfrak{k}$ (see \cite{Moment} for a derivation of $\Phi$). Note that the K\"{a}hler structure on $\mathbb{P}(\mathfrak{g})$ restricts to a $K$-invariant K\"{a}hler structure on the smooth subvariety $\mathbb{P}(\Theta)$, and the action of $K$ on $\mathbb{P}(\Theta)$ is Hamiltonian.

It should be noted that $\mathbb{P}(\Theta)$ is generally not projective. However, $\mathbb{P}(\Theta_{\text{min}})$ is the $G$-orbit in $\mathbb{P}(\mathcal{N})$ of minimal dimension, meaning that it is a closed (hence projective) subvariety of $\mathbb{P}(\mathfrak{g})$. This will be crucial to our study of $\mathbb{P}(\Theta_{\text{min}})$, and subsequently to our description of $\Theta_{\text{min}}$ itself.

\subsection{Description of the $T$-Fixed Points}\label{Fixed}
Let us take a moment to examine the Hamiltonian action of $T$ on $\mathbb{P}(\Theta)$, where $\Theta\subseteq\mathfrak{g}$ is a non-zero nilpotent orbit. We have $$\mathfrak{g}=\mathfrak{h}\oplus\bigoplus_{\beta\in\Delta}\mathfrak{g}_{\beta},$$ the weight space decomposition of the representation $\Adj\vert_{T}$. Note that a point in $\mathbb{P}(\mathfrak{g})$ is fixed by $T$ if and only if it is a class of vectors in $\mathfrak{g}\setminus\{0\}$ with the property that $T$ acts by scaling each vector. In other words, $$\mathbb{P}(\mathfrak{g})^T=\mathbb{P}(\mathfrak{h})\cup\{\mathfrak{g}_{\beta}:\beta\in\Delta\}.$$ With this description, we may determine $\mathbb{P}(\Theta)^T$. Indeed, since $\mathfrak{h}$ consists of semisimple elements of $\mathfrak{g}$ while $\Theta$ consists of non-zero nilpotent elements, we find that $\mathfrak{h}\cap\Theta=\emptyset$. Hence, $$\mathbb{P}(\Theta)^T=\{\mathfrak{g}_{\beta}:\beta\in\Delta, \mathfrak{g}_{\beta}\cap\Theta\neq\emptyset\},$$ a finite set. In particular, $\mathbb{P}(\Theta)^T$ is non-empty if and only if $\Theta$ is the orbit of a root vector.

Let us take a moment to provide a more refined description of $\mathbb{P}(\Theta)^T$. To this end, we will require the below lemma.

\begin{lemma}\label{Weyl}
Let $\beta,\gamma\in\Delta$ be roots. The root spaces $\mathfrak{g}_{\beta}$ and $\mathfrak{g}_{\gamma}$ are $G$-conjugate if and only if $\beta$ and $\gamma$ are conjugate under $W$.
\end{lemma}

\begin{proof}
Suppose that $w\in W$ and that $\beta=w\cdot\gamma$. Choosing a representative $g\in N_G(H)$ of $w$, this means precisely that $\beta=\gamma\circ\varphi_{g^{-1}}\vert_H$, where $\varphi_{g^{-1}}:G\rightarrow G$ is conjugation by $g^{-1}$. Given $h\in H$ and $\xi\in\mathfrak{g}_{\beta}$, note that $$\Adj_{h}(\Adj_g(\xi))=\Adj_g(\Adj_{g^{-1}hg}(\xi))$$ $$=\Adj_g(\beta(g^{-1}hg)\xi)$$ $$=\Adj_g(\gamma(h)\xi)$$ $$=\gamma(h)(\Adj_g(\xi)).$$ It follows that $\mathfrak{g}_{\gamma}=\Adj_g(\mathfrak{g}_{\beta})$.

Conversely, suppose that $g\in G$ and that $\mathfrak{g}_{\gamma}=\Adj_g(\mathfrak{g}_{\beta})$. Consider the Zariski-closed subgroup $$L:=\{x\in G:\Adj_x(\mathfrak{g}_{\gamma})=\mathfrak{g}_{\gamma}\},$$ noting that $H, gHg^{-1}\subseteq L$. Since $H$ and $gHg^{-1}$ are maximal tori of $L$, there exists $x\in L$ for which $xHx^{-1}=gHg^{-1}$. Hence, $x^{-1}g\in N_G(H)$ and $\Adj_{x^{-1}g}(\mathfrak{g}_{\beta})=\mathfrak{g}_{\gamma}$. We may therefore assume that $g\in N_G(H)$. Now, let $w\in W$ denote the class of $g$. Given $h\in H$ and $\xi\in\mathfrak{g}_{\beta}$, we find that $$(w\cdot\beta)(h)\xi=\beta(g^{-1}hg)\xi$$ $$=\Adj_{g^{-1}hg}(\xi)$$ $$=\Adj_{g^{-1}}(\gamma(h)\Adj_g(\xi))$$ $$=\gamma(h)\xi.$$ It follows that $\gamma=w\cdot\beta$.
\end{proof}

Since $\mathfrak{g}$ is a simple Lie algebra, the root system associated with the pair $(\mathfrak{g},\mathfrak{h})$ is irreducible. Hence, there are at most two distinct root lengths (namely, those of the long and short roots), and the roots of a given length constitute an orbit of $W$ in $\Delta$. By Lemma \ref{Weyl}, there are at most two nilpotent $G$-orbits $\Theta$ for which $\mathbb{P}(\Theta)^T$ is non-empty, the orbits of root vectors for the short and long roots. Furthermore, if $\Theta$ is the orbit of a root vector $e_{\beta}\in\mathfrak{g}_{\beta}\setminus\{0\}$, $\beta\in\Delta$, then $\mathbb{P}(\Theta)^T$ is the union of the points $\mathfrak{g}_{\gamma}$ for all $\gamma\in\Delta$ with length equal to that of $\beta$. Since $\Theta_{\text{min}}$ is the orbit of a long root vector, $\mathbb{P}(\Theta_{\text{min}})^T=\{\mathfrak{g}_{\gamma}:\gamma\in\Delta_{\text{long}}\}$, where $\Delta_{\text{long}}\subseteq\Delta$ is the set of long roots.

\subsection{The Moment Polytope of $\mathbb{P}(\Theta_{\text{min}})$}\label{Polytope}
Note that the moment map $\Phi:\mathbb{P}(\mathfrak{g})\rightarrow\mathfrak{k}^*$ considered in \ref{Action} can be modified to obtain a moment map for the Hamiltonian action of $T$ on $\mathbb{P}(\Theta_{\text{min}})$. Indeed, we denote by $\mu:\mathbb{P}(\Theta_{\text{min}})\rightarrow\mathfrak{t}^*$ the moment map given by the composition $$\mathbb{P}(\Theta_{\text{min}})\hookrightarrow\mathbb{P}(\mathfrak{g})\xrightarrow{\Phi}\mathfrak{k}^*\rightarrow\mathfrak{t}^*.$$

Recall that $$\mathbb{P}(\Theta_{\text{min}})^T=\{\mathfrak{g}_{\beta}:\beta\in\Delta_{\text{long}}\}.$$ Given $\beta\in\Delta_{\text{long}}$, choose a point $e_{\beta}\in\mathfrak{g}_{\beta}\setminus\{0\}$. Note that for $X\in\mathfrak{t}$, $$\mu(\mathfrak{g}_{\beta})(X)=\frac{\text{Im}(\langle[X,e_{\beta}],e_{\beta}\rangle)}{\langle e_{\beta},e_{\beta}\rangle}=\frac{\text{Im}(d_e\beta(X)\langle e_{\beta},e_{\beta}\rangle)}{\langle e_{\beta},e_{\beta}\rangle}=\text{Im}(d_e\beta(X)),$$ where $d_e\beta:\mathfrak{t}\rightarrow i\mathbb{R}$ is the morphism of real Lie algebras induced by $\beta:T\rightarrow\Un(1)$. If one regards the weight lattice $\Hom(T,\Un(1))$ as included into $\mathfrak{t}^*$ in the usual way, then our above calculation takes the form $$\mu(\mathfrak{g}_{\beta})=\beta.$$ The moment polytope $\mu(\mathbb{P}(\Theta_{\text{min}}))$ is then the convex hull of $\Delta_{\text{long}}$ in $\mathfrak{t}^*$.

\subsection{Partial Flag Varieties as GKM Manifolds}\label{Partial Flag}
Let us consider the matter of computing $H_{T}^*(\mathbb{P}(\Theta_{\text{min}}))$. To this end, choose a long root $\alpha\in\Delta_{\text{long}}$, so that $\mathfrak{g}_{\alpha}\in\mathbb{P}(\Theta_{\text{min}})^T$. Let $Q$ denote the $G$-stabilizer of $\mathfrak{g}_{\alpha}$. Since $G/Q\cong\mathbb{P}(\Theta_{\text{min}})$ is projective, $Q$ is a parabolic subgroup of $G$. Accordingly, we will address the more general issue of computing the $T$-equivariant cohomology of the partial flag variety $G/P$, where $P\subseteq G$ is a parabolic subgroup containing $T$. Indeed, we will establish that $G/P$ is a GKM (Goresky-Kottwitz-MacPherson) manifold, allowing us to subsequently deploy some well-known machinery to compute its $T$-equivariant cohomology (see \cite{ECB} and \cite{GKM}).

Let us recall the definition of a GKM manifold.
\begin{definition}\label{GKM}
A compact $T$-manifold $X$ is called a GKM manifold if
\begin{itemize}
\item[(i)] $X^T$ is finite, and
\item[(ii)] for every codimension-one subtorus $S\subseteq T$, $\dim(X^S)\leq 2$.
\end{itemize}
\end{definition}

Let us briefly address the significance of this notion in the context of computing $T$-equivariant cohomology. Suppose that $X$ is a GKM manifold as in Definition \ref{GKM}. If $S\subseteq T$ is a subtorus of codimension one and $Y$ is a connected component of $X^S$, then $Y\cap X^T\neq\emptyset$. In particular, $Y$ is $T$-invariant. Furthermore, $Y=\{*\}$ or $Y$ is isomorphic as a $T$-manifold to $S^2$ on which $T$ acts via some non-trivial character $\alpha_{Y}\in\Hom(T,\Un(1))$. In the latter case, $Y^T$ consists of two points, $x_Y^+$ and $x_Y^{-}$.

Let $\{Y_j\}_{j=1}^n$ be the collection of those two-spheres in $X$ arising as connected components of fixed point submanifolds of codimension-one subtori (henceforth called distinguished two-spheres). The inclusion $X^T\hookrightarrow X$ induces an injective graded algebra morphism $H_T^*(X)\hookrightarrow H_T^*(X^T)=\Map(X^T,H_T^*(\text{pt}))$ with image $$\{f\in\Map(X^T,H_T^*(\text{pt})):\forall j\in\{1,\ldots,n\}, \alpha_{Y_j}\vert(f(x_{Y_j}^+)-f(x_{Y_j}^{-}))\}\cong H_T^*(X).$$

Note that Definition \ref{GKM} is precisely the definition of GKM manifold given in \cite{GHZ}, where the authors exhibited certain homogeneous spaces of a compact connected simply-connected semisimple Lie group as GKM manifolds. Below is a statement of their result.

\begin{theorem}\label{GHZTheorem}
Let $M$ be a compact connected simply-connected semisimple Lie group. Let $R\subseteq M$ be a maximal torus, and let $U$ be a closed subgroup of $M$ containing $R$. Assume that $M/U$ is oriented. Then, the left-multiplicative action of $R$ renders $M/U$ a GKM manifold.
\end{theorem}

For the duration of this section, let us fix a parabolic subgroup $P\subseteq G$ satisfying $B\subseteq P$. Note that $P$ is then the standard parabolic subgroup $P_{\Lambda}$ generated by $B$ and the root subgroups $\{U_{-\beta}:=\overline{\exp(\frak{g}_{-\beta})}:\beta\in\Lambda\}$ for some unique subset $\Lambda$ of $\Pi$.

\begin{corollary}
The partial flag variety $G/P$ is a GKM manifold for the left-multiplicative action of $T$.
\end{corollary}

\begin{proof}
The Iwasawa decomposition of $G$ tells us that $G=KB$. In particular, $K$ acts transitively on $G/P$. Since the $K$-stabilizer of the identity coset $[e]\in G/P$ is $K\cap P$, we have a $K$-manifold isomorphism $K/(K\cap P)\cong G/P$. It will therefore suffice to establish that $K/(K\cap P)$ is a GKM manifold for the left-multiplicative action of $T$. For this, we will invoke Theorem \ref{GHZTheorem}. We need only note that $K$ is connected, simply-connected, and semisimple (since $G$ is), that $T\subseteq K\cap P$, and that $K/(K\cap P)$ is oriented (as $G/P$ is).
\end{proof}

It thus remains to determine the fixed points $(G/P)^T$ and the distinguished two-spheres. Accordingly, we will require the below analogue of Theorem 2.2 of \cite{GHZ}.

\begin{lemma}\label{cover}
Let $S\subseteq T$ be a subtorus. The image of $(G/B)^S$ under the fibration $G/B\xrightarrow{\varphi}G/P$ is $(G/P)^S$.
\end{lemma}

\begin{proof}
Consider the fibration $\psi:K/T\rightarrow K/(K\cap P)$. By Theorem 2.2 of \cite{GHZ}, $\psi((K/T)^S)=(K/(K\cap P))^S$. Since each of the maps in the commutative diagram
$$\begin{tikzpicture}
\tikzset{node distance=2.5cm,auto}
  \node (A) {$K/T$};
  \node (B) [below of=A]{$G/B$};
  \node (C) [right of=A] {$K/(K\cap P)$};
  \node (D) [right of=B] {$G/P$};
  \draw[->] (A) to node {$\cong$} (B);
  \draw[->] (A) to node {$\psi$} (C);
  \draw[->] (B) to node {$\varphi$} (D);
  \draw[->] (C) to node {$\cong$} (D);
\end{tikzpicture}$$
is $T$-equivariant, the desired result follows.
\end{proof}

We immediately obtain a description of $(G/P)^T$. Indeed, note that $(G/B)^T=\{[k]:k\in N_K(T)\}$. Hence, $(G/P)^T$ is identified with $N_K(T)/(N_K(T)\cap P)\cong W/W_P$, where $W_P$ is the subgroup of $W$ generated by the simple reflections $\{s_{\beta}:\beta\in\Lambda\}$ (see \cite{BGG}). Let us now determine the distinguished two-spheres in $G/P$.

\begin{lemma}
A submanifold $X\subseteq G/P$ is a distinguished two-sphere if and only if it is related by the action of $N_K(T)$ to a distinguished two-sphere containing the identity coset $[e]$.
\end{lemma}

\begin{proof}
Suppose that $X$ is a two-sphere arising as a component of $(G/P)^S$ for some codimension-one subtorus $S\subseteq T$. Note that $X^T=\{[k_1],[k_2]\}$ for some $k_1,k_2\in N_K(T)$. Furthermore, $R:=k_1^{-1}Sk_1$ is a codimension-one subtorus of $T$ and $k_1^{-1}X\cong S^2$ is a component of $(G/P)^R$ containing $[e]$. The proof of the converse is then a simple reversal of this argument.
\end{proof}

Accordingly, we will temporarily restrict our attention to the distinguished two-spheres in $G/P$ containing $[e]$. Let $X\subseteq G/P$ denote one such two-sphere. Note that $$T_{[e]}(G/P)\cong\mathfrak{g}/\mathfrak{p}\cong\bigoplus_{\beta\in\Delta\setminus\Delta_P}\mathfrak{g}_{\beta}$$ as complex $T$-modules, where $\Delta_P$ is the set of roots whose root spaces belong to $\mathfrak{p}$. Since $T_{[e]}X$ is a complex one-dimensional $T$-invariant subspace of $T_{[e]}(G/P)$, $T_{[e]}X\cong\mathfrak{g}_{\beta}$ for some $\beta\in\Delta\setminus\Delta_P$. In \cite{GHZ}, it is then concluded that $X^T=\{[e],[s_{\beta}]\}\subseteq W/W_P$. Furthermore, by associating to $X$ the weight of $T_{[e]}X$, we obtain a bijection between $\Delta\setminus\Delta_P$ and the distinguished two-spheres containing $[e]$.

Given $[w]\in W/W_P$, choose a representative $k\in N_K(T)$ of $w$. Note that the distinguished two-spheres containing $[w]$ are then the left-translates by $k$ of the distinguished two-spheres containing $[e]$.

\begin{lemma}
Let $X\subseteq G/P$ be a distinguished two-sphere containing $[e]$, so that $Y:=kX$ is a distinguished two-sphere containing $[w]$. If $\beta\in\Delta$ is the weight with which $T$ acts on $T_{[e]}X$, then $w\cdot\beta$ is the weight with which $T$ acts on $T_{[w]}Y$.
\end{lemma}

\begin{proof}
Consider the automorphism $\phi:G/P\rightarrow G/P$, $[g]\mapsto[kg]$, noting that $(d_{[e]}\phi)\vert_{T_{[e]}X}:T_{[e]}X\rightarrow T_{[w]}Y$ is a complex vector space isomorphism. Furthermore, $\phi(t[g])=(ktk^{-1})\phi([g])$ for all $t\in T$ and $g\in G$, so that $d_{[e]}\phi((k^{-1}tk)v)=td_{[e]}\phi(v)$ for all $v\in T_{[e]}(G/P)$. Hence, if $u\in T_{[w]}Y$, then $u=d_{[e]}\phi(v)$ for some $v\in T_{[e]}X$ and $$tu=td_{[e]}\phi(v)$$ $$=d_{[e]}\phi((k^{-1}tk)v)$$ $$=d_{[e]}\phi(\beta(k^{-1}tk)v)$$ $$=\beta(k^{-1}tk)d_{[e]}\phi(v)$$ $$=(w\cdot\beta)(t)u$$ for all $t\in T$.
\end{proof}

Let us summarize our findings.

\begin{theorem}\label{Equivariant Cohomology}
\begin{itemize}
\item[(i)] There is a natural bijection $W/W_P\cong (G/P)^T$.
\item[(ii)] Fix $[w]\in W/W_P\cong (G/P)^T$. Given $\beta\in\Delta\setminus\Delta_P$, there exists a unique distinguished two-sphere $X\subseteq G/P$ with $X^T=\{[w],[ws_{\beta}]\}$, and with the property that $w\cdot\beta$ is the weight of $T_{[w]}X$. Every distinguished two-sphere containing $[w]$ arises in this way.
\item[(iii)] We have a graded algebra isomorphism \begin{align}
H_T^*(G/P)\cong\{& f\in\Map(W/W_P\rightarrow H_{T}^*(\text{pt})):(w\cdot\beta)\vert(f([w])-f([ws_{\beta}]))\nonumber \\ & \forall{w\in W},\beta\in\Delta\setminus\Delta_P\}\nonumber \end{align}.
\end{itemize}
\end{theorem}

\subsection{A Description of $\Theta_{\text{min}}$ and $\mathbb{P}(\Theta_{\text{min}})$}\label{Description}
We devote this section to explicit descriptions of $\Theta_{\text{min}}$ and $\mathbb{P}(\Theta_{\text{min}})$ as homogeneous $G$-varieties. As noted earlier, the latter space is $G$-equivariantly isomorphic to a partial flag variety $G/P$. Accordingly, we shall begin by finding a parabolic subgroup $P\subseteq G$ with this property. In order to proceed, however, we will require the below result.

\begin{theorem}
Let $\Phi$ be an irreducible root system with collection of simple roots $\Sigma\subseteq\Phi$.
\begin{itemize}
\item[(i)] There exists a unique maximal root $\beta\in\Phi$ (called the highest root).
\item[(ii)] This root is long.
\item[(iii)] We have $\langle\beta,\gamma\rangle\geq0$ for all $\gamma\in\Sigma$.
\end{itemize}
\end{theorem}

For a proof, the reader might refer to Propositions 19 and 23 in \cite{Sternberg}.

Denote by $\alpha\in\Delta_{+}$ the highest root, and choose a root vector $e_{\alpha}\in\mathfrak{g}_{\alpha}\setminus\{0\}$. Note that $[e_{\alpha}]=\mathfrak{g}_{\alpha}\in\mathbb{P}(\Theta_{\text{min}})^T$. Let $C_{\mathfrak{g}}(e_{\alpha})$ denote the centralizer of $e_{\alpha}$ with respect to the adjoint representation of $\mathfrak{g}$.

\begin{lemma}
$C_{\mathfrak{g}}(e_{\alpha})$ is a $\mathfrak{t}$-submodule of $\mathfrak{g}$.
\end{lemma}

\begin{proof}
This is a straightforward application of the Jacobi identity. Indeed, suppose that $X\in\mathfrak{t}$ and $Y\in C_{\mathfrak{g}}(e_{\alpha})$. Note that $$[[X,Y],e_{\alpha}]=[X,[Y,e_{\alpha}]]-[Y,[X,e_{\alpha}]]$$ $$=-d_e\alpha(X)[Y,e_{\alpha}]=0.$$
\end{proof}

In other words, $C_{\mathfrak{g}}(e_{\alpha})$ is a sum of $\mathfrak{t}$-submodules of the $\mathfrak{t}$-weight spaces occurring in the adjoint representation of $\mathfrak{t}$ on $\mathfrak{g}$. The summand coming from the trivial weight space $\mathfrak{h}$ is just $\aker(d_e\alpha)$, where we regard $d_e\alpha$ as belonging to $\mathfrak{h}^*$ instead of $\mathfrak{t}^*$. Furthermore, if $\beta\in\Delta$, then $\mathfrak{g}_{\beta}\subseteq C_{\mathfrak{g}}(e_{\alpha})$ if and only if $[\mathfrak{g}_{\alpha},\mathfrak{g}_{\beta}]=\{0\}$. Hence, we have established that $$C_{\mathfrak{g}}(e_{\alpha})=\aker(d_e\alpha)\oplus\bigoplus_{\{\beta\in\Delta:[\frak{g}_{\alpha},\frak{g}_{\beta}]=\{0\}\}}\mathfrak{g}_{\beta}.$$

Now, let $C_G(e_{\alpha})$ and $Q:=C_G([e_{\alpha}])$ be the $G$-stabilizers of $e_{\alpha}\in\Theta_{\text{min}}$ and $[e_{\alpha}]\in\mathbb{P}(\Theta_{\text{min}})$, respectively. The inclusion $C_G(e_{\alpha})\subseteq Q$ yields an inclusion of Lie algebras $C_{\mathfrak{g}}(e_{\alpha})\subseteq\mathfrak{q}:=\Lie(Q)$. Since $\dim_{\mathbb{C}}\mathfrak{q}=\dim_{\mathbb{C}}C_{\mathfrak{g}}(e_{\alpha})+1$ (a consequence of comparing the dimensions of $\Theta_{\text{min}}$ and $\mathbb{P}(\Theta_{\text{min}}))$, and since $\mathfrak{h}\subseteq\mathfrak{q}$ (as $H$ stabilizes $[e_{\alpha}]$), we must have $$\mathfrak{q}=\mathfrak{h}\oplus\bigoplus_{\{\beta\in\Delta:[\frak{g}_{\alpha},\frak{g}_{\beta}]=\{0\}\}}\mathfrak{g}_{\beta}.$$

In light of our having chosen $\alpha$ to be the highest root, $[\frak{g}_{\alpha},\frak{g}_{\beta}]=\{0\}$ for all $\beta\in\Delta_{+}$. It thus remains to determine those negative roots whose root spaces appear as summands of $\mathfrak{q}$.

\begin{lemma}
If $\beta\in\Delta_{-}$, then $[\frak{g}_{\alpha},\frak{g}_{\beta}]=\{0\}$ if and only if $\langle\alpha,\beta\rangle=0$.
\end{lemma}

\begin{proof}
Suppose that $[\frak{g}_{\alpha},\frak{g}_{\beta}]=\{0\}$. Choose $h_{\beta}\in[\frak{g}_{\beta},\frak{g}_{-\beta}]$ such that $d_e\alpha(h_{\beta})=\langle\alpha,\beta\rangle$. Also, select $e_{\beta}\in\frak{g}_{\beta}$ and $f_{\beta}\in\frak{g}_{-\beta}$ such that $h_{\beta}=[e_{\beta},f_{\beta}]$. By assumption, $[e_{\beta},e_{\alpha}]=0$. Since $\alpha$ is the highest root, we also have $[f_{\beta},e_{\alpha}]=0$. Hence, $$0=[e_{\beta},[f_{\beta},e_{\alpha}]]-[f_{\beta},[e_{\beta},e_{\alpha}]]$$ $$=[[e_{\beta},f_{\beta}],e_{\alpha}]$$ $$=[h_{\beta},e_{\alpha}]$$ $$=d_e\alpha(h_{\beta})e_{\alpha}$$ $$=\langle\alpha,\beta\rangle e_{\alpha}.$$ Therefore, $\langle\alpha,\beta\rangle=0$.

Conversely, suppose that $\langle\alpha,\beta\rangle=0$. It will suffice to prove that $\alpha+\beta$ is not a weight of the adjoint representation. Since these weights are $W$-invariant, it will actually suffice to prove that $s_{\beta}(\alpha+\beta)$ is not a weight of $\frak{g}$. However, the orthogonality assumption implies that $s_{\beta}(\alpha+\beta)=\alpha-\beta$. Also, $\alpha-\beta>\alpha$, meaning that $\alpha-\beta$ cannot be a weight of $\frak{g}$.
\end{proof}

Now, suppose that $$\beta=\sum_{\gamma\in\Pi}a_{\gamma\beta}\gamma,$$ $a_{\gamma\beta}\in\mathbb{Z}_{\leq 0}$, is the expression of $\beta$ as a linear combination of simple roots. Since $\langle\alpha,\gamma\rangle\geq 0$ for all $\gamma\in\Pi$, we see that $\langle\alpha,\beta\rangle=0$ if and only if $\langle\alpha,\gamma\rangle=0$ whenever $a_{\gamma\beta}\neq 0$. In other words, $\langle\alpha,\beta\rangle=0$ if and only if $\beta$ is a linear combination of those simple roots orthogonal to $\alpha$.

Accordingly, let us set $$\Xi:=\{\beta\in\Pi:\langle\alpha,\beta\rangle=0\}.$$ We have shown that $Q=P_{\Xi}$, the parabolic subgroup of $G$ determined by the simple roots in $\Xi$.

We thus have the below result concerning the $G$-variety structure of $\mathbb{P}(\Theta_{\text{min}})$.

\begin{theorem}
There is a $G$-variety isomorphism $\mathbb{P}(\Theta_{\text{min}})\cong G/P_{\Xi}$.
\end{theorem}

Let us now address the $G$-variety structure of $\Theta_{\text{min}}$. To this end, we denote by $\mathcal{L}\xrightarrow{\pi}\mathbb{P}(\mathfrak{g})$ the tautological line bundle over $\mathbb{P}(\mathfrak{g})$. Recall that for $\xi\in\mathfrak{g}\setminus\{0\}$, we have $\pi^{-1}([\xi])=\vspan_{\mathbb{C}}\{\xi\}$. Furthermore, the tautological bundle is $G$-equivariant, with the $G$-action on the total space $\mathcal{L}$ given by $$g:([\xi],v)\mapsto([\Adj_g(\xi)],\Adj_g(v)),$$ $g\in G$, $\xi\in\mathfrak{g}\setminus\{0\}$, $v\in\vspan_{\mathbb{C}}\{\xi\}$.

Let $\mathcal{E}\xrightarrow{\varphi}\mathbb{P}(\Theta_{\text{min}})$ denote the pullback of $\mathcal{L}$ along the inclusion $\mathbb{P}(\Theta_{\text{min}})\hookrightarrow\mathbb{P}(\mathfrak{g})$. Note that $\mathcal{E}$ inherits from $\mathcal{L}$ the structure of a $G$-equivariant line bundle over $\mathbb{P}(\Theta_{\text{min}})$. Furthermore, $\Theta_{\text{min}}$ $G$-equivariantly (and also $\mathbb{C}^*$-equivariantly) includes into $\mathcal{E}$ as a smooth open subvariety, namely the complement $\mathcal{E}^*$ of the zero-section. Accordingly, we will describe $\Theta_{\text{min}}$ by more closely examining $\mathcal{E}$.

Since $\mathbb{P}(\Theta_{\text{min}})$ is the homogeneous $G$-variety $G/P_{\Xi}$, we may exhibit $\mathcal{E}$ as an associated bundle for the one-dimensional $P_{\Xi}$-representation $\varphi^{-1}([e_{\alpha}])=\mathfrak{g}_{\alpha}$. More precisely, let $G\times_{P_{\Xi}}\mathfrak{g}_{\alpha}$ denote the quotient of $G\times\mathfrak{g}_{\alpha}$ by the equivalence relation $$(gp,v)\sim(g,\Adj_p(v)),$$ $p\in P_{\Xi}$, $g\in G$, $v\in\mathfrak{g}_{\alpha}$. Consider the map $G\times_{P_{\Xi}}\mathfrak{g}_{\alpha}\rightarrow G/P_{\Xi}$ given by projection from the first component, whose fibres are then naturally complex vector spaces. The bundle $G\times_{P_{\Xi}}\mathfrak{g}_{\alpha}\rightarrow G/P_{\Xi}$ is $G$-equivariant by virtue of the left-multiplicative $G$-action on the first component of $G\times_{P_{\Xi}}\mathfrak{g}_{\alpha}$.

We have an isomorphism $\mathcal{E}\cong G\times_{P_{\Xi}}\mathfrak{g}_{\alpha}$ of $G$-equivariant holomorphic line bundles over $G/P_{\Xi}$, where we are regarding $\mathcal{E}$ as a line bundle over $G/P_{\Xi}$. We therefore have the below description of $\Theta_{\text{min}}$.

\begin{theorem}
There is an isomorphism of $G$-equivariant holomorphic principal $\mathbb{C}^*$-bundles over $G/P_{\Xi}$ between $\Theta_{\text{min}}$ and $(G\times_{P_{\Xi}}\mathfrak{g}_{\alpha})^*$.
\end{theorem}

\subsection{The $T$-Equivariant Cohomology of $\Theta_{\text{min}}$}
Let us use the description of $\Theta_{\text{min}}$ provided in \ref{Description} to compute $H_T^*(\Theta_{\text{min}})$. To this end, we have the equivariant Thom-Gysin sequence $$\cdots\rightarrow H_T^{i-2}(G/P_{\Xi})\rightarrow H_T^{i}(G\times_{P_{\Xi}}\mathfrak{g}_{\alpha})\rightarrow H_T^i((G\times_{P_{\Xi}}\mathfrak{g}_{\alpha})^*)\rightarrow\cdots$$ associated with the zero-section $G/P_{\Xi}$ in $G\times_{P_{\Xi}}\mathfrak{g}_{\alpha}$ and its complement $(G\times_{P_{\Xi}}\mathfrak{g}_{\alpha})^*$. We can say considerably more about this sequence in our context, but it will require a brief computation of the $T$-equivariant Euler class $\text{Eul}_T(N)\in H^2_T(G/P_{\Xi})$ of the normal bundle $N\cong G\times_{P_{\Xi}}\mathfrak{g}_{\alpha}$ of the zero-section in $G\times_{P_{\Xi}}\mathfrak{g}_{\alpha}$. Indeed, we will give the restriction $\text{Eul}_T(G\times_{P_{\Xi}}\mathfrak{g}_{\alpha})\vert_{[w]}\in H_T^2(\text{pt})\cong\Sym^1(\Hom(T,\Un(1))\otimes_{\mathbb{Z}}\mathbb{Q})$ to each fixed point $[w]\in W/W_{P_{\Xi}}\cong (G/P_{\Xi})^T$.

\begin{lemma}\label{Euler}
If $w\in W$, then $\text{Eul}_T(G\times_{P_{\Xi}}\mathfrak{g}_{\alpha})\vert_{[w]}=w\cdot\alpha$.
\end{lemma}

\begin{proof}
Let $i_{[w]}:\{[w]\}\hookrightarrow G/P_{\Xi}$ be the inclusion, and let $i_{[w]}^*:H^*_T(G/P_{\Xi})\rightarrow H^*_T(\text{pt})$ be the associated map on equivariant cohomology. Note that $$\text{Eul}_T(G\times_{P_{\Xi}}\mathfrak{g}_{\alpha})\vert_{[w]}=i_{[w]}^*(\text{Eul}_T(G\times_{P_{\Xi}}\mathfrak{g}_{\alpha}))$$ $$=\text{Eul}_T((i_{[w]})^*(G\times_{P_{\Xi}}\mathfrak{g}_{\alpha}))$$ $$=\text{Eul}_T((G\times_{P_{\Xi}}\mathfrak{g}_{\alpha})_{[w]}),$$ where $(G\times_{P_{\Xi}}\mathfrak{g}_{\alpha})_{[w]}$ is the fibre over $[w]$. Now, choose a representative $k\in N_K(T)$ of $w$, noting that any element of this fibre is of the form $[(k,\xi)]$, $\xi\in\mathfrak{g}_{\alpha}$. Note that for $t\in T$, $$t\cdot[(k,\xi)]=[(tk,\xi)]=[(k(k^{-1}tk),\xi)]$$ $$=[(k,(k^{-1}tk)\cdot\xi)]$$ $$=[(k,\alpha(k^{-1}tk)\xi)]$$ $$=(w\cdot\alpha)(t)[(k,\xi)].$$ Hence, $w\cdot\alpha=\text{Eul}_T((G\times_{P_{\Xi}}\times\mathfrak{g}_{\alpha})_{[w]})=\text{Eul}_T(G\times_{P_{\Xi}}\mathfrak{g}_{\alpha})\vert_{[w]}$.
\end{proof}

In particular, the image of $\text{Eul}_T(G\times_{P_{\Xi}}\mathfrak{g}_{\alpha})$ in $H_T^*((G/P_{\Xi})^T)$ is non-zero. Since restriction gives an inclusion of $H_T^*(G/P_{\Xi})$ into $H_T^*((G/P_{\Xi})^T)$ as a subalgebra, and since $H_T^*((G/P_{\Xi})^T)$ (a direct sum of polynomial rings) has no zero-divisors, we conclude that $\text{Eul}_T(G\times_{P_{\Xi}}\mathfrak{g}_{\alpha})$ is not a zero-divisor in $H_T^*(G/P_{\Xi})$. It follows that our Thom-Gysin sequence splits into the short-exact sequences $$0\rightarrow H_T^{i-2}(G/P_{\Xi})\rightarrow H_T^{i}(G\times_{P_{\Xi}}\mathfrak{g}_{\alpha})\rightarrow H_T^i((G\times_{P_{\Xi}}\mathfrak{g}_{\alpha})^*)\rightarrow 0.$$ (For a proof, see \cite{Yang-Mills}.)

For a second useful refinement of our Thom-Gysin sequence, we note that restriction to the zero-section gives a $T$-equivariant homotopy equivalence between $G\times_{P_{\Xi}}\mathfrak{g}_{\alpha}$ and $G/P_{\Xi}$. It follows that the associated restriction map $H_T^*(G\times_{P_{\Xi}}\mathfrak{g}_{\alpha})\rightarrow H_T^*(G/P_{\Xi})$ is an isomorphism. Using this isomorphism, we shall replace $H_T^*(G\times_{P_{\Xi}}\mathfrak{g}_{\alpha})$ in our short-exact sequences to obtain
$$0\rightarrow H_T^{i-2}(G/P_{\Xi})\rightarrow H_T^{i}(G/P_{\Xi})\rightarrow H_T^i((G\times_{P_{\Xi}}\mathfrak{g}_{\alpha})^*)\rightarrow 0.$$ The map $H_T^{i-2}(G/P_{\Xi})\rightarrow H_T^{i}(G/P_{\Xi})$ is multiplication by $\text{Eul}_T(G\times_{P_{\Xi}}\mathfrak{g}_{\alpha})$ (see \cite{Bifet}, for instance). Furthermore, the map $H_T^{i}(G/P_{\Xi})\rightarrow H_T^i((G\times_{P_{\Xi}}\mathfrak{g}_{\alpha})^*)$ is the map $\psi^*$ on equivariant cohomology induced by the projection $\psi:(G\times_{P_{\Xi}}\mathfrak{g}_{\alpha})^*\rightarrow G/P_{\Xi}$. (This follows from the fact that the bundle projection $G\times_{P_{\Xi}}\mathfrak{g}_{\alpha}\rightarrow G/P_{\Xi}$ and zero-section $G/P_{\Xi}\rightarrow G\times_{P_{\Xi}}\mathfrak{g}_{\alpha}$ give inverse maps on equivariant cohomology.)

The above analysis yields two immediate corollaries. Firstly, the $T$-equivariant Betti numbers $b_T^i(\Theta_{\text{min}})$ of $\Theta_{\text{min}}$ are given by $$b^i_T(\Theta_{\text{min}})=b^i_T(G/P_{\Xi})-b^{i-2}_T(G/P_{\Xi}).$$ Secondly, $\psi^*:H_T^*(G/P_{\Xi})\rightarrow H_T^*(\Theta_{\text{min}})$ is a surjective graded algebra morphism. Its kernel is $\langle\text{Eul}_T(G\times_{P_{\Xi}}\mathfrak{g}_{\alpha})\rangle$, the ideal of $H_T^*(G/P_{\Xi})$ generated by the equivariant Euler class $\text{Eul}_T(G\times_{P_{\Xi}}\mathfrak{g}_{\alpha})\in H^2_T(G/P_{\Xi})$. In particular, there is a graded algebra isomorphism $$H_T^*(\Theta_{\text{min}})\cong H_T^*(G/P_{\Xi})/\langle\text{Eul}_T(G\times_{P_{\Xi}}\mathfrak{g}_{\alpha})\rangle.$$

Using Lemma \ref{Euler} and Theorem \ref{Equivariant Cohomology}, and noting that $W_{P_{\Xi}}=:W_{\Xi}$ is the subgroup of $W$ generated by the reflections $\{s_{\beta}\}_{\beta\in\Xi}$, we obtain the below more explicit description of $H_T^*(\Theta_{\text{min}})$.

\begin{theorem}\label{Minimal Cohomology}
$H_T^*(\Theta_{\text{min}})$ is isomorphic to the quotient of \begin{align}\{& f\in\Map(W/W_{\Xi},H_{T}^*(\text{pt})):(w\cdot\beta)\vert(f([w])-f([ws_{\beta}]))\nonumber\\ & \forall{w\in W},\beta\in\Delta_{-},\text{ }\langle\alpha,\beta\rangle\neq 0\}\nonumber \end{align}
by the ideal generated by the map $W/W_{\Xi}\rightarrow H_T^*(\text{pt})$, $[w]\mapsto w\cdot\alpha$.
\end{theorem}

\subsection{An Example}
Let us compute the equivariant cohomology of the minimal nilpotent orbit of $G=\SL_2(\mathbb{C})$. To this end, let $T\subseteq G$ be the compact real form of the standard maximal torus of $G$. Note that $\Delta=\{-2,2\}\subseteq\mathbb{Z}\cong\Hom(T,\Un(1))$ is the resulting collection of roots. Letting $B\subseteq G$ be the Borel subgroup of upper-triangular matrices, we find that $\alpha=2$ is the highest root. It is not orthogonal to any of the simple roots, so that $\Xi=\emptyset$. Hence, $P_{\Xi}=B$ and $\Delta_{P_{\Xi}}=\{2\}$. The Weyl group $W$ is $\mathbb{Z}/2\mathbb{Z}$, and the generator acts by negation on the weight lattice. The subgroup $W_{\Xi}$ is trivial. In particular, $G/P_{\Xi}$ has two $T$-fixed points.

Since $\alpha$ is identified with $2x\in\mathbb{Q}[x]\cong H_T^*(\text{pt})$, Theorem \ref{Equivariant Cohomology} implies that $H_T^*(G/P_{\Xi})$ includes into $H_T^*(\text{pt})^{\oplus 2}\cong\mathbb{Q}[x]^{\oplus 2}$ as the subalgebra $$H_T^*(G/P_{\Xi})\cong\{(f_1(x),f_2(x))\in\mathbb{Q}[x]^{\oplus 2}:x\vert(f_1(x)-f_2(x))\}$$ $$=\{(f_1(x),f_2(x))\in\mathbb{Q}[x]^{\oplus 2}:f_1(0)=f_2(0)\}.$$ Indeed, we have recovered the $\Un(1)$-equivariant cohomology of the two-sphere with the rotation action of $\Un(1)$.

Lemma \ref{Euler} tells us that $\text{Eul}_T(N)=(2x,-2x)$ when included into $\mathbb{Q}[x]^{\oplus 2}$. Hence, $$H_T^*(\Theta_{\text{min}})\cong\frac{\{(f_1(x),f_2(x))\in\mathbb{Q}[x]^{\oplus 2}:f_1(0)=f_2(0)\}}{\langle(x,-x)\rangle}.$$

Note that this is generated as a $\mathbb{Q}$-algebra by $y:=[(x,0)]$. The relation is $y^2=0$, so that $$H_T^*(\Theta_{\text{min}})\cong\mathbb{Q}[y]/\langle y^2\rangle,$$ with $y$ an element of grading degree two.

We remark that this is consistent with Corollary \ref{flagvariety}. Indeed, if $G=\SL_2(\mathbb{C})$, then $\Theta_{\text{min}}=\Theta_{\text{reg}}$. Hence, $H_T^*(\Theta_{\text{min}})=H_T^*(\Theta_{\text{reg}})$. Corollary \ref{flagvariety} tells us that the latter is isomorphic to the ordinary cohomology of $G/B\cong\mathbb{P}^1$.

\end{document}